\documentclass{amsart}

\title[Automorphism Groups]{Automorphism Groups of Countably Categorical Linear Orders are Extremely Amenable}
\author{
  Fran\c{c}ois Dorais
  \and Steven Gubkin
  \and Daniel McDonald
  \and Manuel Rivera}
\date{ }
\thanks{The research of the first author was partly supported by NSF grant DMS-0648208.}
\usepackage{verbatim}
\usepackage{enumerate}
\usepackage{amsmath}
\usepackage{amssymb}
\usepackage{amsthm}

\newcounter{cnt}[section]

\newtheorem{thm}[cnt]{Theorem}
\newtheorem{prp}[cnt]{Proposition}
\newtheorem{lem}[cnt]{Lemma}
\newtheorem{cor}[cnt]{Corollary}
\theoremstyle{definition}

\newcommand*{\term}[1]{{\bfseries#1}}

\newcommand{\IFF}{\Leftrightarrow}
\newcommand{\lthen}{\rightarrow}
\newcommand{\liff}{\leftrightarrow}
\newcommand{\embed}{\hookrightarrow}

\newcommand{\cat}{\mathop{\hat{}}}
\newcommand{\restrict}{{\upharpoonright}}
\newcommand*{\seq}[1]{\left\langle#1\right\rangle}
\newcommand*{\set}[1]{\left\lbrace#1\right\rbrace}

\DeclareMathOperator{\Aut}{Aut}
\DeclareMathOperator{\Age}{Age}

\newcommand{\EQ}{\mathrel{E}}
\newcommand{\NEQ}{\not\mathrel{E}}

\newcommand{\Q}{\mathbb{Q}}
\newcommand{\LO}[1]{\mathbb{L}(#1)}
\newcommand{\LX}[1]{\widehat{\mathbb{L}}(#1)}
\newcommand{\FC}[1]{\mathcal{K}_{#1}}

\begin{document}

\maketitle

\begin{abstract}
  We show that the automorphism groups of countably categorical linear
  orders are extremely amenable. Using methods of Kechris, Pestov, and
  Todorcevic, we use this fact to derive a structural Ramsey theorem
  for certain families of finite ordered structures with finitely many
  partial equivalence relations with convex classes.
\end{abstract}

\section{Introduction}

A \term{countably categorical linear order} is a countable (possibly
finite) linear order $L$ such that every countable linear order which
satisfies the same first-order theory as $L$ (in the language of
linear order) is isomorphic to $L$.  A topological group $G$ is
\term{extremely amenable} if every action of $G$ on a compact
Hausdorff space has a fixed point.  In~\cite{Pestov98}, Pestov showed
that the automorphism group of the linear order of the rationals $\Q$
--- which is a countably categorical linear order --- is extremely
amenable.  In this paper, we generalize this result to the class of
all countably categorical linear orders.

\begin{thm}\label{thm:extreme}
  The automorphism group of a countably categorical linear order is
  extremely amenable.
\end{thm}

Our proof of this theorem uses a characterization of the countably categorical
linear orders due to Rosenstein~\cite{Rosenstein69}.  Given an indexed
family $\seq{\tau_i}_{i \in I}$ of linear order types where the index
set $I$ is a linear order, the generalized sum $\sum_{i \in I} \tau_i$
is the order type of the order obtained by concatenating intervals of type
$\tau_i$ along the ordering given by $I$. Two special types of
generalized sums will be of interest to us:
\begin{description}
\item[Finite sums] $s(\tau_1,\dots,\tau_n) = \sum_{i \in
    \set{1,\dots,n}} \tau_i$ where $\set{1,\dots,n}$ has the usual
  ordering.
\item[Finite shuffles] $\sigma(\tau_1,\dots,\tau_n) = \sum_{q \in \Q}
  \tau_q$ where each $D_i = \set{ q \in \Q : \tau_q = \tau_i}$ is a
  dense subset of $\Q$ and $\Q = D_1 \cup \cdots \cup D_n$.
\end{description}
See~\cite{Rosenstein82} for more on these operations.  It turns out
that these two operations generate all countably categorical linear
order types.

\begin{thm}[Rosenstein~\cite{Rosenstein69}]\label{thm:rosenstein}
  The class of countably categorical linear order types is the
  smallest class of linear order types that contains the order type
  $1$ and is closed under finite sums and finite shuffles.
\end{thm}

\noindent
Therefore, every countably categorical linear order type can be
described by a sequence of applications of sums $s(T_1,\dots,T_k)$ and
shuffles $\sigma(T_1,\dots,T_k)$ of arbitrary arity --- a
\term{sum-shuffle expression}.  In general, a countably categorical
order type will have many sum-shuffle expressions, for example, the
expressions $\sigma(1)$, $s(\sigma(1),\sigma(1))$,
$\sigma(\sigma(1))$, and $\sigma(s(1,\sigma(1)))$ all represent the
order type of the rational numbers. One of our main results
(Theorem~\ref{thm:gubkincan}) shows every countably categorical linear
order has a canonical sum-shuffle expression that best captures
the structure of the linear order.

To prove Theorem~\ref{thm:extreme}, we will first associate to each
sum-shuffle expression $T$ a linear order $\LX{T}$ of type $T$
expanded with finitely many relation symbols such that $\Aut(\LX{T})$
is easily computed and shown to be extremely amenable
(Theorem~\ref{thm:gubkinaut}). Next we will show that every countably
categorical linear order has a canonical sum-shuffle expression $T$
such that $\Aut(\LO{T}) = \Aut(\LX{T})$, where $\LO{T}$ is the
underlying linear order of $\LX{T}$ (Theorem~\ref{thm:gubkincan}).

The analysis of $\LX{T}$ will also show that $\LX{T}$ can be seen as
the limit of a Fra{\"\i}ss{\'e} class $\FC{T}$ of finite ordered
structures. As a consequence of general results of Kechris, Pestov and
Todorcevic~\cite{KechrisPestovTodorcevic05}, it follows that these
Fra\"{\i}ss\'{e} classes $\FC{T}$ all have the Ramsey property. This
gives an infinite family of structural Ramsey theorems
(Corollary~\ref{cor:ramseyprop}).

Some instances of this family of structural Ramsey theorems correspond
to existing results. First, the structural Ramsey theorem for the
class $\FC{\sigma(1)}$ is nothing but a thinly disguised form of
Ramsey's theorem~\cite{Ramsey30}. The case $\FC{\sigma(\sigma(1))}$
equivalent to a partition theorem of Rado~\cite{Rado54} (see
also~\cite[Corollary~6.8]{KechrisPestovTodorcevic05}).  More
generally, the case $\FC{\sigma^n(1)}$ corresponds to the fact that
the Fra{\"\i}ss{\'e} order class $\mathcal{U}_S^{c{<}}$ of finite
convexly ordered ultrametric spaces with distances in a fixed
$n$-element subset $S$ of $(0,\infty)$ has the Ramsey property, which
is a result of Nguyen Van Th{\'e}~\cite{NguyenVanThe09}. Thus the classes
$\FC{T}$ are combinatorially very rich and can be used to encode a
variety of natural Fra{\"\i}ss{\'e} order classes with the Ramsey
property.

\section{Tree Presentations and Coordinatization}\label{sec:axioms}

While sum-shuffle expression are easy to understand, we will find it
more convenient to work with parse trees for such expressions. These
parse trees will be represented as trees of sequences of positive
integers with labels from $\set{\ell,s,\sigma}$, for \term{leaf},
\term{sum}, \term{shuffle}, respectively. Formally, we define
\term{tree presentations} via the following inductive rules.
\begin{itemize}
\item $1$ is a tree presentation that consists only of the root
  $\seq{}$, with label $\ell$.
\item If $T_1,\dots,T_k$ are tree presentations, then so is
  $s(T_1,\dots,T_k)$ whose nodes are the root $\seq{}$, with label
  $s$, and nodes of the form $\seq{i}\cat{t}$ for $t \in T_i$
  (including the root $t = \seq{}$) and $i \in \set{1,\dots,k}$, with
  the same label as that of $t$ in $T_i$.
\item If $T_1,\dots,T_k$ are tree presentations, then so is
  $\sigma(T_1,\dots,T_k)$ whose nodes are the root $\seq{}$, with
  label $\sigma$, and nodes of the form $\seq{i}\cat{t}$ for $t \in
  T_i$ (including the root $t = \seq{}$) and $i \in \set{1,\dots,k}$,
  with the same label as that of $t$ in $T_i$.
\end{itemize}
Note that only \emph{leaf nodes} (childless nodes of the tree) have label
$\ell$ and all remaining nodes have label $s$ or $\sigma$. If $T$ is a
tree presentation, we say that a linear order $L$ has \emph{type $T$} if
the order-type of $L$ is the evaluation of the sum-shuffle expression
corresponding to $T$.

For each tree presentation $T$, we construct a canonical linear order
$\LO{T}$ with type $T$. To do this, we fix, once and for all, a
partition $\seq{\Q_n}_{n=1}^{\infty}$ of $\Q$, each part of which is
dense in $\Q$. This way, for each $n \geq 1$, we will have a canonical
dense partition $\seq{\Q_1,\dots,\Q_n}$ of the dense linear order
$\Q_{[n]} = \Q_1 \cup\cdots\cup \Q_n$.  For convenience, let us
further require that $n \in \Q_n$ for each positive integer $n$. For
each rational $q \in \Q$, let $\#(q)$ be the unique positive integer
such that $q \in \Q_{\#(q)}$. Note that $\#(n) = n$ for every positive
integer $n$.

The canonical linear order $\LO{T}$ will be a
lexicographically ordred set of finite sequences of rationals which
is \emph{prefix-free} (no element of the set is an initial segment of another).
To determine whether a sequence $\bar{r} = \seq{r_0,\dots,r_{k-1}}$ of
rationals belongs to $\LO{T}$, write $\#(\bar{r}) =
\seq{\#(r_0),\dots,\#(r_{k-1})}$, then check that $\#(\bar{r})$ is a
leaf-node of $T$ and, for each $i < k$, make sure that if
$\#(\bar{r})\restrict i$ is a sum-node of $T$ then $r_i =
\#(r_i)$. (There is nothing further to check when
$\#(\bar{r})\restrict i$ is a shuffle-node.)

\begin{prp}\label{prp:mcdonaldcrd}
  For each tree presentation $T$, $\LO{T}$ is a linear order of
  type~$T$.
\end{prp}

By construction, $\LO{T}$ has some structural properties that are not
always captured by the order relation alone. For each $t \in T$, let
the \term{$t$-domain $D_t$} be the set of all elements $\bar{r}$ of
$\LO{T}$ such that the leaf-node $\#(\bar{r})$ extends the node $t$ in
$T$. In particular, $D_{\seq{}} = \LO{T}$ where $\seq{}$ is the root
of $T$.

Each element $\bar{r}$ of $D_t$ is contained in a unique
\term{$t$-interval}: the maximal interval $I$ of $\LO{T}$ such that
$\bar{r} \in I \subseteq D_t$. In fact, it is easy to see that for
each $\bar{r} \in D_t$, the $t$-interval containing $\bar{r}$ is
\begin{equation*}
  I_t(\bar{r}) =
  \set{\bar{s} \in \LO{T} : \bar{s}\restrict|t| = \bar{r}\restrict|t|}.
\end{equation*}
In particular, $I_{\seq{}}(\bar{r}) = \LO{T}$ for every $\bar{r} \in
D_{\seq{}} = \LO{T}$, and if $t$ is a leaf-node of $T$, then
$I_t(\bar{r}) = \set{\bar{r}}$ for every $\bar{r} \in D_t$.

We will now enumerate a list of universal axioms for a theory that
will capture the fine structure of $\LO{T}$. In addition to the order
${<}$ and equality ${=}$ relations, we expand our language to contain
one binary relation ${\EQ}_t$ for each node $t$ of $T$. The intended
interpretation in $\LO{T}$ is $\bar{r} \EQ_t \bar{s}$ if and only if $\bar{r}$
and $\bar{s}$ belong to the same $t$-interval; let $\LX{T}$ denote
$\LO{T}$ expanded with these binary relations.  These relations
$\EQ_t$ are \emph{partial equivalence relations} (symmetric,
transitive, but not necessarily reflexive relations) with convex
equivalence classes. To talk about $t$-domains in the language of
$\LX{T}$, simply note that $\bar{r} \in D_t \IFF \bar{r} \EQ_t
\bar{r}$ since $\bar{r} \in D_t$ if and only if $\bar{r}$ belongs to
some $t$-interval of $\LO{T}$.

The universal axioms characterizing $\LX{T}$ are naturally divided
into five groups.

\begin{enumerate}[\bfseries(T1)]
\item \label{tax:equiv} For each node $t$ of $T$, the following are
  axioms:
  \begin{gather*}
    x \EQ_t y \lthen y \EQ_t x, \\
    x \EQ_t y \land y \EQ_t z \lthen x \EQ_t z, \\
    x \EQ_t y \land x < z \land z < y \lthen x \EQ_t z \land z \EQ_t
    y.
  \end{gather*}
  In other words, every ${\EQ_t}$ is a partial equivalence relation
  with convex classes (the $t$-intervals).
\item \label{tax:root} The following is an axiom:
  \begin{equation*}
    x \EQ_{\seq{}} y
  \end{equation*}
  In other words, there is a unique $\seq{}$-interval which consists
  of every point.
\item \label{tax:leaf} For every leaf $t$ of $T$, the following is an
  axiom:
  \begin{equation*}
    x \EQ_t y \lthen x = y
  \end{equation*}
  In other words, $t$-intervals consist of only one point.
\item \label{tax:child} If $t$ is a sum or shuffle node with $k$
  children in $T$, then the following are axioms:
  \begin{gather*}
    x \EQ_t x \liff x \EQ_{t\cat\seq{1}} x \lor \cdots \lor x \EQ_{t\cat\seq{k}} x,\\
    x \NEQ_{t\cat\seq{i}} x \lor x \NEQ_{t\cat\seq{j}} x, \quad
    \text{for $1 \leq i < j \leq k$.}
  \end{gather*}
  In other words, the $t$-domain is the disjoint union of the
  $t\cat\seq{i}$-domains.
\item \label{tax:sum} If $t$ is a sum node with $k$ children in $T$,
  then the following are axioms:
  \begin{gather*}
    x \EQ_t y \land x \EQ_{t\cat\seq{i}} x \land y \EQ_{t\cat\seq{i}}
    y \lthen x \EQ_{t\cat\seq{i}} y, \\
    x \EQ_t y \land x \EQ_{t\cat\seq{i}} x \land y \EQ_{t\cat\seq{j}}
    y \lthen x < y, \quad \text{for $1 \leq i < j \leq k$.}
  \end{gather*}
  In other words, each $t$-interval is a finite union of consecutive
  $t\cat\seq{i}$-intervals.
\end{enumerate}
It is a simple matter to check that $\LX{T}$ satisfies all of these
axiom groups.  More importantly, these axioms characterize the
substructures of $\LX{T}$.

\begin{prp}\label{prp:mcdonaldext} 
  A countable structure $L$ satisfies the axioms of
  {\upshape(T\ref{tax:equiv}--T\ref{tax:sum})} if and only if it is
  embeddable into $\LX{T}$. In fact, if $K$ is a finite substructure
  of $L$, then any embedding $K\embed\LX{T}$ can be extended to an
  embedding $L\embed\LX{T}$.
\end{prp}

\begin{proof}
  It suffices to show that if $x \in L$, $K$ is a finite substructure
  of $L$, and $e:K\embed\LX{T}$ is an embedding, then $e$ can be
  extended to an embedding $\bar{e}:K \cup \{x\}\embed\LX{T}$.

  By~(T\ref{tax:child}), we see that there is a unique leaf-node $t$
  of $T$ such that $D_t(x)$.  Let $\ell = |t|$ and define the
  coordinates $\bar{e}(x)(0),\dots,\bar{e}(x)(\ell-1)$ in order as
  follows.
  \begin{itemize}
  \item If $t\restrict{i}$ is a sum-node, then $\bar{e}(x)(i)$ must match
  the $i$-th coordinate of $t$.
  \item If $t\restrict{i}$ is a shuffle-node and there is a $y \in K$
    such that $x \EQ_{t\restrict(i+1)} y$, then $\bar{e}(x)(i)$ must
    match $e(y)(i)$.
  \item If $t\restrict{i}$ is a shuffle-node and there is no $y \in K$
    such that $x \EQ_{t\restrict(i+1)} y$, then we are free to choose any
    $\bar{e}(x)(i) \in \Q_{t(i)}$ such that $\bar{e}(x)(i) < e(z)(i)
    \IFF x < z$ for all $z \in K$ such that $x \EQ_{t\restrict{i}} z$.
  \end{itemize}
  Since $e:K\embed\LX{T}$ is an embedding, any choice of $y$
  in the second case will give the same value for $\bar{e}(x)(i)$.
  Similarly, in the third case, the fact that $\Q_{t(i)}$ is dense ensures that a
  suitable value for $\bar{e}(x)(i)$ can always be found.

  This completes the definition of the extended map
  $\bar{e}:K\cup\{x\}\embed\LX{T}$; we need to check that this is
  indeed an embedding, i.e., that $\bar{e}$ preserves the partial
  equivalence relations ${\EQ_{t'}}$ and the order relation ${<}$.
  
  \textit{Preservation of the partial equivalence relations.}  Since
  $e = \bar{e} \restrict K$ is known to be an embedding, it suffices
  to check that $x \EQ_{s} y \IFF \bar{e}(x) \EQ_{s} \bar{e}(y)$ for
  every $y \in K$.  First, note that if $s$ is not among
  $t\restrict{0} = \seq{},t\restrict{1},\dots,t\restrict{\ell} = t$,
  then $\lnot D_t(x)$ and hence $\lnot D_{s}(\bar{e}(x))$. Therefore,
  $x \NEQ_{s} y$ and $\bar{e}(x) \NEQ_{s} \bar{e}(y)$ for any $y \in
  K$. So it suffices to show that the relations
  ${\EQ_{t\restrict{i}}}$ are preserved, for $i = 0,\dots,\ell$.

  By~(T\ref{tax:root}) and~(T\ref{tax:child}), we know that for every
  $y \in K$ there is a maximal $i \leq \ell$ such that $x
  \EQ_{t\restrict{i}} y$. It suffices to show that $\bar{e}(x)
  \EQ_{t\restrict{i}} \bar{e}(y)$ and, provided $i < \ell$, that
  $\bar{e}(x) \NEQ_{t\restrict(i+1)} \bar{e}(y)$.

  By~(T\ref{tax:child}), we know that $x \EQ_{t\restrict{j}} y$ for
  every $j < i$. By the above construction, we see that at every stage
  $j < i$ where $t\restrict{j}$ is a sum-node, we picked
  $\bar{e}(x)(j) = t(j) = \bar{e}(y)(j)$.  Also, at every stage $j <
  i$ where $t\restrict{j}$ is a shuffle-node, we explicitly picked
  $\bar{e}(x)(j) = \bar{e}(y)(j)$. Therefore, $\bar{e}(x)(j) =
  e(y)(j)$ for every $j < i$, which implies that $\bar{e}(x)
  \EQ_{t\restrict{i}} \bar{e}(y)$.

  Suppose now that $i < \ell$. We want to show that $\bar{e}(x)
  \NEQ_{t\restrict(i+1)} \bar{e}(y)$. We consider two cases.
  \begin{itemize}
  \item If $t\restrict{i}$ is a sum-node, then the first axiom of
    group (T\ref{tax:sum}) implies that $\lnot
    D_{t\restrict(i+1)}(y)$. Therefore $\lnot
    D_{t\restrict(i+1)}(\bar{e}(y))$, which implies that $\bar{e}(x)
    \NEQ_{t\restrict(i+1)} \bar{e}(y)$.
  \item If $t\restrict{i}$ is a shuffle node, then either $x
    \EQ_{t\restrict(i+1)} z$ for some $z \in K$, and $\bar{e}(x)
    \EQ_{t\restrict(i+1)} \bar{e}(z)$ by definition of
    $\bar{e}(x)$. However, we must then have $y \NEQ_{t\restrict(i+1)}
    z$ and hence $\bar{e}(y) \NEQ_{t\restrict(i+1)}
    \bar{e}(z)$. Otherwise, $\bar{e}(x)(i)$ was chosen to be either
    strictly smaller or strictly bigger than $\bar{e}(z)(i)$ for every
    $z \in K$ such that $z \EQ_{t\restrict{i}} x$. In particular,
    $\bar{e}(x)(i) \neq \bar{e}(y)(i)$ which means that $\bar{e}(x)
    \NEQ_{t\restrict(i+1)} \bar{e}(y)$.
  \end{itemize}

  \textit{Preservation of the order relation.}  Since $e = \bar{e}
  \restrict K$ is known to be an embedding and the order relation on
  $\LX{T}$ is total, it suffices to check that $x = y \IFF \bar{e}(x)
  = \bar{e}(y)$ and $x < y \IFF \bar{e}(x) < \bar{e}(y)$ for every $y
  \in K$.  By (T\ref{tax:root}), (T\ref{tax:leaf}), and
  (T\ref{tax:child}) it is true that $x=y \IFF x \EQ_{t\restrict{i}}
  y$ for every $i = 0,\dots,\ell$.  Since $\bar{e}$ is known to
  preserve the partial equivalence relations, for every $y \in K$ we
  have $x=y \IFF \bar{e}(x) = \bar{e}(y)$ and $x\neq y$ if and only if the
  maximal $i$ such that $x \EQ_{t\restrict{i}} y$ satisfies $i <
  \ell$.
  
  Suppose $x<y$.  Let $i$ be the maximal number less than $\ell$ such
  that $x \EQ_{t\restrict{i}} y$ and let
  $(t\restrict{i})\cat\seq{1},\dots,(t\restrict{i})\cat\seq{k}$ be all
  the children of $t\restrict{i}$.  Thus by (T\ref{tax:child}) we can
  let $(t\restrict{i})\cat\seq{m}$ and $(t\restrict{i})\cat\seq{n}$ be
  the two distinct children of $t\restrict{i}$ such that
  $D_{(t\restrict{i})\cat\seq{m}}(x)$ (i.e.,
  $t\restrict(i+1)=(t\restrict{i})\cat\seq{m}$) and
  $D_{(t\restrict{i})\cat\seq{n}}(y)$.
  
  If $t\restrict{i}$ is a sum-node then $m<n$ by (T\ref{tax:sum}).
  Since $\bar{e}$ is known to preserve the partial equivalence
  relations, $\bar{e}(x) \EQ_{t\restrict{i}} \bar{e}(y) \land
  D_{(t\restrict{i})\cat\seq{m}}(\bar{e}(x)) \land
  D_{(t\restrict{i})\cat\seq{n}}(\bar{e}(y))$.  Thus $\bar{e}(x) <
  \bar{e}(y)$ by (T\ref{tax:sum}).
   
  If $t\restrict{i}$ is a shuffle-node then by our construction
  $\bar{e}(x)(j) = \bar{e}(y)(j)$ for all $j=0,\dots,i-1$ but
  $\bar{e}(x)(i) < \bar{e}(y)(i)$.  Therefore $\bar{e}(x) <
  \bar{e}(y)$ because the ordering is lexicographic.
\end{proof}

\section{Extreme Amenability of $\Aut(\LX{T})$}\label{sec:extreme}

In this section, we will establish the first step in the proof of
Theorem~\ref{thm:extreme}.

\begin{thm}\label{thm:gubkinaut} 
  For every tree presentation $T$, the automorphism group of $\LX{T}$
  is extremely amenable.
\end{thm}

We proceed by induction on the structure of $T$. The result is trivial
for $T = 1$ since the automorphism group of $\LX{1}$ is the trivial
group, which is clearly extremely amenable.  To complete the
induction, it suffices to show that if
$\Aut(\LX{T_1}),\dots,\Aut(\LX{T_k})$ are extremely amenable, then so
are $\Aut(\LX{s(T_1,\dots,T_k)})$ and
$\Aut(\LX{\sigma(T_1,\dots,T_k)})$.

To handle sums, we make the following simple observation.

\begin{lem}
  If $T = s(T_1,\dots,T_k)$ then
  \begin{equation*}
    \Aut(\LX{T}) \cong \Aut(\LX{T_1}) \times \cdots \times \Aut(\LX{T_k}).
  \end{equation*}
\end{lem}

\noindent 
Since extremely amenable groups are closed under products
\cite[Lemma~6.7]{KechrisPestovTodorcevic05}, it follows that if
$\Aut(\LX{T_1}),\dots,\Aut(\LX{T_k})$ are extremely amenable then so
is $\Aut(\LX{s(T_1,\dots,T_k)})$.

Shuffles require a more subtle argument. We begin with this
observation, which the main part of the proof of
\cite[Lemma~8.4]{KechrisPestovTodorcevic05}.

\begin{lem}
  For every positive integer $k$, the group $H_k$ is extremely
  amenable, where $H_k$ consists of all order automorphisms of
  $\Q_{[k]} = \Q_1 \cup \cdots \cup \Q_k$ that preserve each $\Q_i$
  setwise.
\end{lem}

\noindent
The heart of the proof is the following key fact.

\begin{lem}
  If $T = \sigma(T_1,\dots,T_k)$ then $\Aut(\LX{T}) \cong H_k \ltimes
  G$ where
  \begin{equation*}
    G = {\textstyle\prod_{q \in \Q_{[k]}}} \Aut(\LX{T_{\#(q)}})
    \cong \left(\Aut(\LX{T_1}) \times\cdots\times \Aut(\LX{T_k})\right)^\omega
  \end{equation*}
  and $H_k$ acts on $G$ by permuting the index set $\Q_{[k]} = \Q_1
  \cup \cdots \cup \Q_k$.
\end{lem}
  
\begin{proof}
  For each $q \in \Q_{[k]}$, let $I_q$ be the interval of $\LX{T}$
  consisting of elements with first coordinate $q$. Note that deleting
  the first coordinate gives a natural isomorphism $I_q \cong
  \LX{T_{\#(q)}}$.
  
  Since every automorphism of $\LX{T}$ maps each interval $I_q$ onto a
  similar interval $I_{q'}$, we have a natural homomorphism
  $h:\Aut(\LX{T}) \to H_k$ defined by the relation
  \begin{equation*}
    h(\alpha)(x(0)) = \alpha(x)(0)
  \end{equation*}
  for all $x \in \LX{T}$. Moreover, $h$ has a right inverse $s:H_k \to
  \Aut(\LX{T})$ where, for each $\pi \in H_k$, $s(\pi)$ acts on the
  first coordinate according to $\pi$ but leaves all other coordinates
  unchanged.
  
  The kernel of $h$ is the set
  \begin{equation*}
    K = \{ \alpha \in \Aut(\LX{T}) : \text{$\alpha(x)(0) = x(0)$ for
      every $x \in \LX{T}$}  \}.
  \end{equation*}
  Thus the restriction of an element $\alpha \in K$ to $I_q$ is an
  automorphism of $I_q$. Pasting these restrictions together and
  piping them through the natural isomorphisms $I_q \cong
  \LX{T_{\#(q)}}$ yields isomorphisms
  \begin{equation*}
    \ker(h) \cong \prod_{q \in \Q_{[k]}} \Aut(I_q) 
    \cong \prod_{q \in \Q_{[k]}} \Aut(\LX{T_{\#(q)}}) = G.
  \end{equation*}
  It follows at once that $\Aut(\LX{T}) \cong H_k \ltimes G$, as
  described in the statement of the lemma.
\end{proof}
  
\noindent
To finish the proof of Theorem~\ref{thm:gubkinaut}, we appeal again to
\cite[Lemma~6.7]{KechrisPestovTodorcevic05} where it is shown that
extremely amenable groups are closed under arbitrary products and
semidirect products.  It follows that if
$\Aut(\LX{T_1}),\dots,\Aut(\LX{T_k})$ are extremely amenable, then so
is $\Aut(\LX{\sigma(T_1,\dots,T_k)})$.

\section{Canonical Tree Presentations}\label{sec:canonical}

In this section, we will establish the final step in the proof of
Theorem~\ref{thm:extreme}.  A tree presentation $T$ is said to be
\term{canonical} if every automorphism of $\LO{T}$ is also an
automorphism of $\LX{T}$, hence $\Aut(\LO{T}) = \Aut(\LX{T})$.  Since
$\Aut(\LX{T})$ is known to be extremely amenable, it follows that
$\Aut(\LO{T})$ is extremely amenable too.

\begin{thm}\label{thm:gubkincan}
  Every countably categorical linear order has a canonical tree
  presentation.
\end{thm}

Here is an outline of the proof. Given a countably categorical linear
order $L$ we will inductively construct a sequence $D_1,\dots,D_k$ of
dense linear orders (possibly with endpoints and possibly
trivial). Each $D_i$ will be equipped with a labeling that assigns to
each point $q \in D_i$ a tree presentation $T_q$. At each stage, we
will have an isomorphism $h_i: L \cong \sum_{q \in D_i} \LO{T_q}$.
The final linear order $D_k$ will be trivial, so $h_k$ will be an
isomorphism from $L$ onto $\LO{T_\star}$, where ${\star}$ is the
unique element of $D_k$.

To ensure that $T_\star$ is a canonical tree presentation for $L$, we
will show that at each stage $h_i$ induces an isomorphism 
\begin{equation*}
  \textstyle
  \hat{h}_i: \Aut(L) \cong G_i \ltimes \prod_{q \in D_i} \Aut(\LX{T_q}),
\end{equation*}
where $G_i$ is the group of automorphisms of $D_i$ that preserve the
labeling $q \in D_i \mapsto T_q$ and $G_i$ acts on $\prod_{q \in D_i}
\Aut(\LX{T_q})$ by permuting the indices.  At the last stage, $G_k$ is
trivial and hence $h_k: L \cong \LO{T_\star}$ induces an isomorphism
$\Aut(L) \cong \Aut(\LX{T_\star})$, which will show that $T_\star$ is
a canonical tree presentation of~$L$.

The method for constructing the dense linear orders $D_1,\dots,D_k$
was developed by Rosenstein. We will appeal
to the proof of~\cite[Theorem~8.40]{Rosenstein82} for some useful facts about the
construction, but we need to recall the main steps of the the construction
in some detail in order to establish the relevant facts about
automorphism groups. Our notation will diverge from
Rosenstein's, but the translation will always be clear.

The first dense linear order $D_1$ is the finite condensation of $L$,
i.e., $D_1$ is the collection of all maximal finite intervals of $L$
with the induced ordering. As observed by Rosenstein, every
element of $L$ is contained in a maximal finite interval of $L$, so we
have a unique isomorphism $h_1:L \cong \sum_{q \in D_1} \LO{T_q}$. The
labeling $q \in D_1 \mapsto T_q$ simply assigns to each maximal finite
interval $q$ the tree presentation of the finite linear order of
length $|q|$. An automorphism of $L$ must map each maximal finite
interval to a maximal finite interval of the same length and thus
corresponds to a unique element of $G_1$. Conversely, any $\alpha \in
G_1$ has a unique expansion to an automorphism of $L$ by mapping each
element of the finite interval $q$ to the corresponding element of
$\alpha(q)$. Since $\Aut(\LX{T_q})$ is trivial for each $q \in D_1$,
this correspondence gives an isomorphism
\begin{equation*}
  \textstyle
  \hat{h}_1:\Aut(L) \cong G_1 \cong G_1 \ltimes \prod_{q \in D_1} \Aut(\LX{T_q}).
\end{equation*}

The next dense linear orders are obtained by a two-stage process. We
first perform the label condensation of $D_i$ with respect to the
labeling $q \in D_i \mapsto T_q$ to obtain a linear order $E_i$.  We
say that an interval $I \subseteq D_i$ is \term{homogeneous} if it has
no endpoints and $\{ T_q : q \in I \} = \{ T_q : r < q < s\}$ holds
for all $r, s \in I$ with $r < s$. The linear order $E_i$ consists
of the collection of all maximal homogeneous intervals of $D_i$
together with all singleton intervals for elements of $D_i$ that are
not contained in any homogeneous interval of $D_i$.

To each $I \in E_i$, we assign a tree presentation $S_I$. Rosenstein
shows that the set $\{T_q : q \in D_i\}$ is always finite, so fix,
once and for all, an enumeration $T_1,\dots, T_n$ of this set. If $I$
is a singleton, say $I = \{q\}$, we simply define $S_I =
T_q$. Otherwise, we assign $S_I = \sigma(T_{i(1)},\dots,T_{i(k)})$
where $i(1) < \cdots < i(k)$ are such that $T_{i(1)},\dots,T_{i(k)}$
enumerates $\{ T_q : q \in I\}$. Note that if $I$ and $J$ are two
maximal homogeneous intervals such that $\{T_q : q \in I\} = \{T_r : r
\in J\}$ then $S_I = S_J$.

It is easy to see that if $I \in E_i$, then $\sum_{q \in I} \LO{T_q} \cong
\LO{S_I}$. Therefore,
\begin{equation*}
  \textstyle
  L \cong \sum_{q \in D_i} \LO{T_q}
  \cong \sum_{I \in E_i} \sum_{q \in I} \LO{T_q} 
  \cong \sum_{I \in E_i} \LO{S_I}.
\end{equation*}
Let $k_i:L \cong \sum_{I \in E_i} \LO{S_I}$ be the isomorphism just
described. There is more than one choice for $k_i$ but any choice
which respects the above decompositions will do. In particular, $k_i$
must be compatible with $h_i$, which realizes the first of these
decompositions, in the sense that $h_i(x) = (q,\bar{r})$ and $k_i(x) =
(I,\bar{s})$, then $q \in I$, if $I$ is a singleton then $\bar{s} =
\bar{r}$, and if $I$ is a maximal homogeneous interval then $\bar{s} =
\seq{s_0}\cat\bar{r}$.

Let $H_i$ be the group of automorphisms of $E_i$ that preserve the
labeling $I \in E_i \mapsto S_I$.  Each $\alpha \in G_i$ must map a
maximal homogeneous interval $I$ of $D_i$ to another maximal
homogeneous interval $\alpha(I)$ in such a way that $S_I =
S_{\alpha(I)}$. Similarly, if $q$ is not contained in any homogeneous
interval of $D_i$ then $\alpha(q)$ has the same property and
$S_{\{q\}} = T_q = T_{\alpha(q)} = S_{\{\alpha(q)\}}$. Therefore, we
have a group homomorphism $g_i:G_i \to H_i$.
Note that this homomorphism $g_i$ has a section $s_i:H_i \to G_i$
where each $\eta \in H_i$ is expanded to $s_i(\eta) \in G_i$ using
$h_i$ to select a canonical isomorphism between $I$ and $\eta(I)$.

The kernel of $g_i$ is the subgroup $K_i = \{\alpha \in G_i : (\forall
I \in E_i)(\alpha(I) = I)\}$. Observe that $G_i \cong H_i \ltimes K_i$
and that
\begin{equation*}
  \textstyle
  K_i \ltimes \prod_{q \in D_i} \Aut(\LX{T_q}) \cong \prod_{I \in E_i}
  \Aut(\LX{S_I}),
\end{equation*}
where $K_i$ acts on $\prod_{q \in D_i} \Aut(\LX{T_q})$ by permuting
the indices. By the induction hypothesis, the isomorphism $h_i$
induces an isomorphism
\begin{equation*}
  \textstyle
  \hat{h}_i:\Aut(L) \cong G_i \ltimes \prod_{q \in D_i} \Aut(\LX{T_q}).
\end{equation*}
It follows from the above computations that the isomorphism $k_i$
similarly induces an isomorphism
\begin{equation*}
  \textstyle
  \hat{k}_i:\Aut(L) \cong H_i \ltimes \prod_{I \in E_i} \Aut(\LX{S_I}).
\end{equation*}

Finally, to obtain $D_{i+1}$, we perform the finite condensation of
$E_i$. Rosenstein shows that every element of $E_i$ is contained in a
maximal finite interval of $E_i$, so $D_{i+1}$ is a well-defined dense
linear order.
To each $q \in D_{i+1}$ we define $T_q =
s(S_{I_1},\dots,S_{I_k})$, where $I_1 < \cdots < I_k$ is the
increasing enumeration of $q$. Any $\alpha \in H_i$ must map a maximal
finite interval $q$ of $E_i$ to a maximal finite interval $\alpha(q)$
of $E_i$ in such a way that $T_q = T_{\alpha(q)}$ and thus $\alpha$
corresponds to a unique element of $G_{i+1}$. Conversely, any $\alpha
\in G_{i+1}$ has a unique expansion to an element of $H_i$ by mapping
each element of the finite interval $q$ to the corresponding element
of $\alpha(q)$. Therefore, $G_{i+1} \cong H_i$ and since
\begin{equation*}
  \Aut(\LX{T_q}) = \Aut(\LX{S_{I_1}}) \times \cdots \times \Aut(\LX{S_{I_k}}),
\end{equation*}
where $I_1 < \cdots < I_k$ is the increasing enumeration of $q \in D_{i+1}$, we
have
\begin{equation*}
  \textstyle
  \prod_{q \in D_{i+1}} \Aut(\LX{T_q}) \cong \prod_{I \in E_i} \Aut(\LX{S_I}).
\end{equation*}
It follows immediately that 
\begin{equation*}
  \textstyle
  \hat{h}_{i+1}:\Aut(L) \cong G_{i+1} \ltimes \prod_{q \in D_{i+1}} \Aut(\LX{T_q}).
\end{equation*}

The only detail that remains is to show that this process must
eventually terminate by reaching a step $k$ where $D_k$ is
trivial --- this termination argument is also given by
Rosenstein.

\section{The Fra{\"\i}ss{\'e} Order Class $\FC{T}$}\label{sec:fraisse}

In Section~\ref{sec:axioms}, we have expanded the linear order
$\LO{T}$ by adding certain binary relations to form the structure
$\LX{T}$. In this section, we will show that this new structure
$\LX{T}$ can be regarded as a Fra{\"\i}ss{\'e} limit of a
Fra{\"\i}ss{\'e} order class $\FC{T}$. Before we prove this, let us
first briefly review the relevant parts of Fra{\"\i}ss{\'e} theory; a
detailed discussion can be found in~\cite{Hodges97}, for example.

Let $L$ be a first-order language with finitely many relation symbols
and no function symbols. A class $\mathcal{K}$ of finite
$L$-structures is a \term{Fra{\"\i}ss{\'e} class} if it satisfies the
following three properties.
\begin{description}
\item[Hereditary Property] If $A \in \mathcal{K}$ and $B \embed A$
  then $B \in \mathcal{K}$.
\item[Joint Embedding Property] If $A,B \in \mathcal{K}$ there is
  $C\in \mathcal{K}$ such that $A \embed C$ and $B \embed C$.
\item[Amalgamation Property] If $A,B,C \in \mathcal{K}$ are such that
  $f:A \embed B$, $g: A \embed C$, there are a $D \in \mathcal{K}$ and
  $f': B \embed D$, $g': C \embed D$ such that $f' \circ f = g' \circ
  g$.
\end{description}
If, moreover, there is a distinguished binary relation symbol ${<}$ in
$L$ which is interpreted as a linear order in every element of
$\mathcal{K}$, then we say that $\mathcal{K}$ is a
\term{Fra{\"\i}ss{\'e} order class}.

If $D$ is any $L$-structure, the \term{age} of $D$ is the class
$\Age(D)$ of finite $L$-structure of structures that can be embedded
into $D$.  It is clear that $\Age(D)$ satisfies the hereditary and
joint embedding properties.  We say that $D$ is
\term{ultrahomogeneous} every isomorphism between finite substructures
of $D$ can be extended to an automorphism of $D$; this condition
guarantees that $\Age(D)$ also satisfies the amalgamation
property. Hence, $\Age(D)$ is a Fra{\"\i}ss{\'e} class whenever $D$ is
an ultrahomogeneous $L$-structure.  The converse of this fact is the
basis of Fra{\"\i}ss{\'e} theory.

\begin{thm}[Fra{\"\i}ss{\'e}~\cite{Fraisse54}]\label{thm:fraisse}
  If $\mathcal{K}$ is a Fra{\"\i}ss{\'e} class of finite
  $L$-structures, then there is a countable ultrahomogeneous
  $L$-structure $D$, unique up to isomorphism, such that $\mathcal{K}$
  is the age of $D$.
\end{thm}

\noindent
The unique countable structure $D$ of Theorem~\ref{thm:fraisse} is
called the \term{Fra{\"\i}ss{\'e} limit} of the class $\mathcal{K}$.

A standard back-and-forth argument using
Proposition~\ref{prp:mcdonaldext} shows that:

\begin{prp}\label{prp:riveracrd}
  For every tree presentation $T$, the expanded structure $\LX{T}$ is
  ultrahomogeneous.
\end{prp}

\noindent
Denote by $\mathcal{K}_T$ the age of $\LX{T}$. It follows from
Theorem~\ref{prp:riveracrd} that $\mathcal{K}_T$ is a Fra{\"\i}ss{\'e}
order class. Thus, by Proposition~\ref{prp:mcdonaldext}, the class
$\mathcal{K}_T$ is precisely the class of finite ordered structures
that satisfy the axioms {\upshape(T\ref{tax:equiv}--T\ref{tax:sum})}
of Section~\ref{sec:axioms}.
 
\begin{thm}\label{thm:fraisselimit} 
  For every tree presentation $T$, the class $\FC{T}$ is a
  Fra{\"\i}ss{\'e} order class and $\LX{T}$ is its Fra{\"\i}ss{\'e}
  limit.
\end{thm}

We will now turn to structural Ramsey theory.  Let $L$ be a
first-order language with finitely many relation symbols and no
function symbols. If $A, B$ are finite $L$-structures, we denote by
$\binom{B}{A}$ the set of all substructures of $B$ which are
isomorphic to $A$. If $C$ is another finite $L$-structure $k$ is a
positive integer, we write
\begin{equation*}
  C\rightarrow (B)^A_k
\end{equation*} 
if for every coloring $\chi: {\binom{C}{A}} \rightarrow \{1,2,...,k\}$
there exists $B' \in {\binom{C}{B}}$ such that the $\chi$ is constant
on ${\binom{B'}{A}}$. We say that the class $\mathcal{H}$ of finite
$L$-structures satisfies the \term{Ramsey property} if for any two
structures $A,B\in \mathcal{H}$ and every positive integer $k$, there
exists $C\in \mathcal{H}$ such that $C\rightarrow (B)^A_k$.  This
property was considered by Kechris, Pestov, and Todorcevic, who
characterized which Fra{\"\i}ss{\'e} order classes satisfy the Ramsey
property as follows.

\begin{thm}[Kechris--Pestov--Todorcevic~\cite{KechrisPestovTodorcevic05}]
  Let $\mathcal{K}$ be a Fra{\"\i}ss{\'e} order class with
  Fra{\"\i}ss{\'e} limit $D$.  Then $\Aut(D)$ is extremely amenable if
  and only if $\mathcal{K}$ has the Ramsey property.
\end{thm}

\noindent
Since $\LX{T}$ is extremely amenable by Theorem~\ref{thm:gubkinaut}
and $\LX{T}$ is the Fra{\"\i}ss{\'e} limit of the Fra{\"\i}ss{\'e}
order class $\FC{T}$, we can apply the above theorem to obtain that
$\FC{T}$ has the Ramsey property.

\begin{cor}\label{cor:ramseyprop}
  For each tree presentation $T$, the Fra{\"\i}ss{\'e} order class
  $\FC{T}$ has the Ramsey property.
\end{cor}

As announced in the introduction, structures $\FC{T}$ can be seen as
convexly ordered ultrametric spaces whose open balls coincide with the
various $t$-classes of the structure. In fact, there is a precise
biinterpretation between $\FC{\sigma^n(1)}$ and the Fra{\"\i}ss{\'e}
order class
$\mathcal{U}_S^{c{<}}$ of finite convexly ordered
ultrametric spaces with distances in a fixed $n$-element set $S
\subseteq (0,\infty)$, as previously considered by Nguyen Van
The~\cite{NguyenVanThe09}.

To see how this correspondence works, suppose $S =
\set{s_1,\dots,s_n}$ where $s_0 = 0 < s_1 < \cdots < s_n$. Similarly,
suppose $t_0,\dots,t_n$ enumerates the tree presentation $\sigma^n(1)$
from the leaf $t_0 = \seq{1,\dots,1}$ to the root $t_n = \seq{}$.
A finite convexly ordered ultrametric space $(A,d,{<})$ with distances
in $S$ can be viewed as an element of $\FC{\sigma^n(1)}$ by definining the relations $x
\EQ_{t_i} y \IFF d(x,y) \leq s_i$.  Conversely, a structure
$(A,{\EQ_{t_0}},{\EQ_{t_0}},\dots,{\EQ_{t_n}},{<})$ in
$\FC{\sigma^n(1)}$ can be made into a convexly ordered ultrametric
space with distances in $S$ by defining $d(x,y) = s_i$ where $i$ is least such that $x \EQ_{t_i} y$.

This back and forth translation gives an equivalence between the
Fra{\"\i}ss{\'e} order classes $\mathcal{U}_S^{c{<}}$ and
$\FC{\sigma^n(1)}$.
Therefore, Corollary~\ref{cor:ramseyprop} gives the following.

\begin{cor}[Nguyen Van Th{\'e}~\cite{NguyenVanThe09}]
  Let $S$ be a finite set of positive real numbers.
  The Fra{\"\i}ss{\'e} order class $\mathcal{U}_S^{c{<}}$ has the
  Ramsey property.
\end{cor}

\noindent
Nguyen Van Th{\'e} further shows that $\mathcal{U}_S^{c{<}}$ has the Ramsey
property even when $S$ is an infinite subset of $(0,\infty)$. This
general case does not correspond to a special case of
Corollary~\ref{cor:ramseyprop}, but one can derive this
more general result from the case where $S$ is finite.

\bibliographystyle{amsplain}
\bibliography{ramsey}

\end{document}